\documentclass[11pt]{amsart}
\usepackage{extarrows}
\usepackage[colorlinks, citecolor=blue]{hyperref}
\usepackage{amsthm,amsmath,amssymb}
\usepackage{mathrsfs}
\usepackage{bold-extra}
\usepackage{algpseudocode}
\usepackage[T1]{fontenc}
 \usepackage[ruled,vlined,linesnumbered]{algorithm2e}
\usepackage{tikz}
\usepackage{tikz-cd}

\newtheorem{theorem}{Theorem}[section]
\newtheorem{lemma}[theorem]{Lemma}
\newtheorem{mainthm}{Theorem}

\theoremstyle{definition}

\newtheorem{proposition}[theorem]{Proposition}
\newtheorem{corollary}[theorem]{Corollary}
\newtheorem{remark}[theorem]{Remark}
\newtheorem{conjecture}[theorem]{Conjecture}

\theoremstyle{remark}

\usepackage[backref=true,backend=biber, style=alphabetic,sorting=nyt,maxnames=99,maxalphanames=99]{biblatex}
\AtBeginBibliography{\scriptsize}

\numberwithin{equation}{section}

\DeclareMathOperator{\HSC}{HSC}
\DeclareMathOperator{\RBC}{RBC}
\DeclareMathOperator{\Sym}{Sym}
\DeclareMathOperator{\tr}{tr}
\DeclareMathOperator{\ad}{ad}
\DeclareMathOperator{\Span}{span}

\DeclareMathOperator{\End}{End}
\DeclareMathOperator{\SU}{SU}

\newcommand*{\dif}{\mathop{}\!\mathrm{d}}

\begin{document}

\title{Torsion algebras of Hermitian manifolds and rigidity}

\makeatletter
\def\shorttitle{Torsion algebras of Hermitian manifolds and rigidity}
\makeatother

\author{Wangyang Lin}

\address{School of Mathematical Sciences, Fudan University, Shanghai 200433, China}

\email{wylin23@m.fudan.edu.cn, wylin\_math@outlook.com}

\author{YiBo Ren}

\address{School of Mathematical Sciences, Fudan University, Shanghai 200433, China}

\email{ybren24@m.fudan.edu.cn}

\subjclass[2020]{53C55, 53B35, 17B30}


\keywords{Complex Manifold, Chern--K\"{a}hler-like, Torsion algebra, Holomorphic sectional curvature}

\begin{abstract}
    We regard the Chern torsion of a Hermitian manifold as a skew-symmetric complex-bilinear product on its holomorphic tangent bundle. For a compact connected Chern--K\"{a}hler-like Hermitian manifold, we prove that this product satisfies the Jacobi identity pointwise. If, in addition, the Chern holomorphic sectional curvature is strongly quasi-positive, we show that the resulting torsion Lie algebra is nilpotent and must be abelian. Consequently, the metric is K\"{a}hler. The underlying complex manifold is therefore projective and rationally connected. For general Hermitian manifolds, we construct a metric with positive real bisectional curvature on a Hopf surface which is neither simply connected nor rationally connected.
\end{abstract}

\maketitle

\tableofcontents

\section{Introduction}

Let $(M^n,J,h)$ be a connected Hermitian manifold of complex dimension $n$,
and write $\omega$ for its fundamental form.
Throughout this paper, $\nabla$ means the Chern connection on its holomorphic tangent bundle
$T^{1,0}M$. 
In a local unitary frame
$e_1,\ldots,e_n$, we write
\[
    R_{i\bar j k\bar\ell}=h\bigl(R^\nabla(e_i,\bar e_j)e_k,e_\ell\bigr).
\]
For $0\ne v\in T^{1,0}_pM$, the Chern holomorphic sectional curvature is
\[
    \HSC_{\nabla}(v)=\frac{R^\nabla(v,\bar v,v,\bar v)}{|v|_h^4}.
\]
If $e=(e_1,\ldots,e_n)$ is unitary and
$a=(a_1,\ldots,a_n)\in\mathbb R^n_{\geq0}\setminus\{0\}$, the (unnormalized) real
bisectional curvature is
\[
 \RBC_{\nabla}(e,a) =\sum_{i,j=1}^n R_{i\bar i j\bar j}a_i a_j.
\]

We say that Chern HSC is \emph{positive} if it is positive at every point. For RBC, we say it is positive if it is positive at every point for every unitary frame and every nonzero nonnegative weight.

We say that Chern HSC is \emph{strongly quasi-positive} if it is nonnegative at every point and every nonzero direction, and if at one point it is positive in every nonzero direction.  Strongly quasi-positive RBC means nonnegativity for every point, every unitary frame, and every nonzero nonnegative weight, with strict positivity for all those data at one point.  
Some papers use ``quasi-positive'' for exactly this condition, while others reserve that term for positivity in one unspecified direction.  
We use the stronger convention because it is the convention used in this paper.

The metric is called \emph{Chern--K\"{a}hler-like} (CKL) when
\begin{equation}\label{eq:intro-ckl}
    R_{i\bar j k\bar\ell}=R_{k\bar j i\bar\ell} \qquad (1\leq i,j,k,\ell\leq n).
\end{equation}
Together with the Hermitian symmetry, this gives the usual K\"{a}hler curvature symmetries for the Chern tensor.  CKL does not mean K\"{a}hler: 
a classical example is the Iwasawa manifold whose Chern connection is flat.  
A basic result of Yang and Zheng \cite{YangZheng2018CurvatureTensors} is that every compact CKL metric (they call it K\"{a}hler-like) is balanced,
\[ 
    d\omega^{n-1}=0, 
\]
although balancedness by itself is much weaker than the K\"{a}hler condition \eqref{eq:intro-ckl}. The sign conditions for Chern HSC and Chern RBC are equivalent, including the strong quasi-positive condition.  
\begin{proposition}[Proposition 1.3~\cite{YangZheng2019}]\label{prop:equiv_HSC_RBC}
    On a CKL Hermitian manifold, strongly quasi-positive Chern HSC is equivalent to strongly quasi-positive Chern RBC.
\end{proposition}

\subsection{Holomorphic bisectional curvature, projectivity, and rational curves}

Holomorphic bisectional curvature was introduced by Goldberg and Kobayashi \cite{GoldbergKobayashi1967}.  
It is stronger than holomorphic sectional curvature.  
The classical Frankel problem asked whether a compact K\"{a}hler manifold with positive holomorphic bisectional curvature must be projective space \cite{Frankel1961}.  Mori's algebro-geometric theorem that a projective manifold with ample tangent bundle is $\mathbb P^n$ \cite{Mori1979}, together with the differential-geometric work of Siu and Yau \cite{SiuYau1980}, settled the Frankel problem:
\[
    \operatorname{HBC}>0+\text{K\"{a}hler}\quad\Longrightarrow\quad M\simeq\mathbb P^n \quad\text{(biholomorphically)}.
\]
Mok classified the nonnegative case: the universal cover is a product of Euclidean factors, compact irreducible Hermitian symmetric factors, and projective-space factors \cite{Mok1988}.    

For HSC, Tsukamoto proved that a compact K\"{a}hler manifold with positive HSC is simply connected \cite{Tsukamoto1957}.  Yau later asked whether positive HSC also forces projectivity and rational connectedness (Problem~47 in his problem list) \cite{Yau1982}.  
Campana and Kollár--Miyaoka--Mori proved that Fano manifolds are rationally connected \cite{Campana1992,KollarMiyaokaMori1992a,KollarMiyaokaMori1992b}.  
Campana--Demailly--Peternell gave a criterion in terms of pseudo-effectivity of tensor differentials, which is particularly suited to curvature-induced vanishing \cite{CampanaDemaillyPeternell2015}.
Kodaira's embedding theorem shows that the vanishing of $h^{2,0}$ for a compact K\"{a}hler manifold implies the existence of a projective polarization in the situations used here \cite{Kodaira1954,MorrowKodaira1971}.

\subsection{The positive and quasi-positive HSC results in the K\"{a}hler category}

The first systematic results for nonnegative HSC addressed birational and Hodge-theoretic restrictions.  Yang proved that a compact Hermitian manifold with semipositive HSC which is not identically flat has Kodaira dimension $-\infty$; he also used this to exclude Hermitian metrics with semipositive HSC on Kodaira and hyperelliptic surfaces \cite{Yang2016}.  
His subsequent work on big vector bundles and semipositive tangent bundles placed these restrictions in the framework of positivity of vector bundles and the birational geometry of compact complex manifolds \cite{Yang2017}.  These results show, in particular, that semipositive HSC does not simply translate into positivity of $c_1(M)$.

Yang introduced RC-positivity for Hermitian holomorphic vector bundles.  For an RC-positive bundle $E$, he proved vanishing of sufficiently high symmetric powers of $E^*$ after tensoring with an arbitrary fixed bundle.  On a compact K\"{a}hler manifold, RC-positivity of all exterior powers of $T^{1,0}M$ implies projectivity and rational connectedness.  He then proved that positive HSC gives this RC-positivity, that
\[
    H^{p,0}(M)=0\quad(1\leq p\leq n),
\]
and consequently that $M$ is projective and rationally connected \cite{Yang2018}.  This confirmed the strict-positive form of Yau's Problem~47.  Uniform RC-positive metrics and their implications for rationally connected manifolds were developed further in \cite{Yang2020}.

The quasi-positive case was established in stages.  Heier and Wong proved rational connectedness for projective K\"{a}hler manifolds with partially positive curvature, in particular for the projective quasi-positive HSC setting; their earlier paper related positive scalar curvature to uniruledness \cite{HeierWong2012,HeierWong2020}.  Matsumura obtained restrictions on the image of the maximal rationally connected (MRC) fibration for projective manifolds with semipositive HSC and then obtained a structural description of projective semipositive-HSC manifolds, including rational connectedness in the quasi-positive case \cite{Matsumura2020,Matsumura2022}.  Zhang and Zhang removed the projectivity assumption: a compact K\"{a}hler manifold with quasi-positive HSC is projective and rationally connected \cite{ZhangZhang2023}.

\subsection{Chern real bisectional curvature}

Yang and Zheng introduced real bisectional curvature in the modern Hermitian
setting.  Their paper defines the curvature, proves that on a Chern--K\"{a}hler-like
metric its sign is equivalent to the sign of HSC, classifies compact metrics
with constant nonzero RBC, describes the zero-curvature consequences, and
extends the Wu--Yau Schwarz lemma to targets with nonpositive RBC
\cite{YangZheng2019}.  They also emphasize that without Chern--K\"{a}hler-like
symmetries RBC and HSC are genuinely different: RBC is then the stronger
condition.  

Several papers investigate positive Chern curvature stronger than HSC. 
Tang studied constant Chern HSC in the K\"{a}hler-like setting and formulated rigidity questions for compact Hermitian metrics with constant curvature \cite{Tang2021}.

\subsection{K\"{a}hler-like connections and compact Hermitian rigidity}

In this paper, ``K\"{a}hler-like'' must always be qualified by the connection: Chern--K\"{a}hler-like, Levi--Civita (or Gray--K\"{a}hler-like), Strominger--Bismut-K\"{a}hler-like, and general Gauduchon K\"{a}hler-like are different conditions.

Yang and Zheng gave the first systematic comparison of the Chern and Levi--Civita curvature symmetries.  They introduced the Chern--K\"{a}hler-like (which they call K\"{a}hler-like) and Gray--K\"{a}hler-like terminology, proved that either condition forces balancedness on a compact manifold, showed that simultaneous satisfaction of
both conditions implies K\"{a}hlerness when the manifold is compact (and also in complex dimension at most three), and proved uniqueness up to scale within a conformal class \cite{YangZheng2018CurvatureTensors}.

Balas and Balas--Gauduchon treated compact Hermitian surfaces with constant Chern HSC, while Davidov, Grantcharov, and Muskarov studied the Chern curvature of twistor spaces \cite{Balas1985,BalasGauduchon1985,DavidovGrantcharovMuskarov2009}. Li and Zheng classified constant-HSC structures on complex nilmanifolds \cite{LiZheng2022}; Rao and Zheng treated pluriclosed manifolds with constant HSC \cite{RaoZheng2022}; and Chen and Zheng studied Strominger space forms,
Bismut-torsion-parallel metrics, and constant HSC for canonical metric connections \cite{ChenZheng2022,ChenZheng2024,ChenZheng2025}.  Broder and Tang considered compact Hermitian manifolds with vanishing curvature \cite{BroderTang2025}.

\subsection{Vanishing of holomorphic tensors}
For vanishing of tensors, the classical Bochner results, the Kobayashi--Wu vector-bundle theorem, and Kobayashi's first-Chern-class results give
\[
 H^0\!\left(M,(T^{1,0}M)^{\otimes p}\otimes (T^{*1,0}M)^{\otimes q}\right)=0 \quad\text{when one tensor weight dominates the other},
\]
under appropriate Ricci or $c_1$ sign conditions
\cite{Bochner1946,Bochner1949,KobayashiWu1970,Kobayashi1980a,Kobayashi1980b}.

For HSC, Yang's RC-positivity theorem gives vanishing of all holomorphic forms under positive HSC in the K\"{a}hler category \cite{Yang2018}.  Ping Li first proved the corresponding tensor vanishing for compact CKL metrics with strictly positive or strictly negative HSC: for a suitable constant $C$,
\[
 q>Cp\quad\Longrightarrow\quad H^0\!\left(M,(T^{1,0}M)^{\otimes p}\otimes (T^{*1,0}M)^{\otimes q}\right)=0
\]
in the positive case, with the dual inequality in the negative case \cite{Li2024}.  His later work introduced uniformly RC $k$-positive bundles, proved a general tensor-vanishing theorem, and applied it to positive $k$-Ricci curvature and K\"{a}hler-like Hermitian metrics \cite{Li2026}.

\subsection{Results of this paper}

We now state the results with the notation fixed above.  

\begin{mainthm}\label{mainthm:CKL-Kahler}
Let $(M^n,J,h)$ be compact and connected, and suppose that the Chern
curvature of $h$ is CKL.  If the Chern holomorphic sectional curvature is
strongly quasi-positive, then the Chern torsion vanishes identically.  Hence
$h$ is K\"{a}hler.
\end{mainthm}

By \cite[Theorem~1.7]{ZhangZhang2023}, we have the following corollary:
\begin{corollary}
If a compact connected CKL Hermitian manifold has strongly quasi-positive
Chern RBC or Chern HSC, then its complex manifold is projective and rationally connected.
\end{corollary}

Moreover, we can prove the following vanishing of mixed-type tensors, thanks to the results of \cite{LiZhangZhang2024}.

\begin{corollary}\label{cor:vanish_pq}
Under the hypotheses of Theorem~\ref{mainthm:CKL-Kahler}, there is a constant $C>0$ such that, for
all nonnegative integers $p$ and $q$,
\[
 p>Cq\quad\Longrightarrow\quad
 H^0\!\left(M,(T^{1,0}M)^{\otimes q}\otimes
 (T^{*1,0}M)^{\otimes p}\right)=0.
\]
In particular, $H^{p,0}(M)=0$ for $1\leq p\leq n$.
\end{corollary}

In \cite[Conjecture~1.6(a)(b)]{YangZheng2019}, the authors conjectured that
\begin{conjecture}
    Let $(M^n,h)$ be a compact connected Hermitian manifold. If $\RBC_\nabla(e,a)>0$ for any unitary frame $e$ and $a\in\mathbb{R}_{\ge 0}^n\backslash\{0\}$, then $M$ is simply connected and rationally connected.
\end{conjecture}

We find a counterexample for this conjecture.
\begin{mainthm}\label{mainthm:counterexample_YZ}
    Let $0<q<1$ and let
    \[
        H_q = (\mathbb{C}^2\backslash\{0\})/\langle z\mapsto qz\rangle
    \]
    be the Hopf surface. There is a positive number $c$ such that for any nonzero $\epsilon\in(-c,c)$, there is a Hermitian metric $h_\epsilon$ on $H_q$ such that its Chern RBC is positive. 
\end{mainthm}
Since positive RBC implies positivity of HSC, this gives a Hermitian manifold with positive HSC but is not simply-connected or rationally connected.

\subsection*{Declaration on the use of AI}

During the exploratory stage of this work, the authors used large language models to suggest possible proof strategies. The resulting suggestions were treated solely as heuristic input and were not accepted without independent mathematical verification. All statements, proofs, and computations included in the final manuscript were independently checked and revised by the authors. The authors take full responsibility for the originality, correctness, and presentation of the work.

\subsection*{Acknowledgement}

Both authors thank Professor Ping Li for proposing this problem and suggestions.

\section{Chern torsion algebra}
Let $D$ be a Hermitian connection on holomorphic tangent bundle $T^{1,0}M$ i.e. $Dh = 0$ and $DJ=0$. The torsion of $D$ is defined as
\[
    T_D(X,Y):=D_XY-D_YX-[X,Y],\quad \text{for any locally smooth vector fields } X,Y,
\]
It is tensorial in $X$ and $Y$, thus we can define  for $x\in M$,
\[
    [u,v]_{D}:=T_D(u,v)_x\in T^{1,0}_xM,\quad \forall u,v\in T^{1,0}_xM.
\]
Then $\mathfrak{A}_{x,D} := (T^{1,0}_xM,[\cdot,\cdot]_{D})$ is a complex anticommutative algebra. If $D$ is the Chern connection $\nabla$, we call $\mathfrak{A}_{x,\nabla}$ the \emph{Chern torsion algebra} of the Hermitian manifold $(M,h)$ at $x$. To the best of our knowledge, the concept of torsion algebra was used in \cite{NiZhengAmbrose2023,wang2026bismuttorsionparallelhermitianmanifoldsconstant}.

We denote the smooth bundle of $(p,q)$ forms $\wedge^{p,q}M$, and especially the holomorphic cotangent bundle $\varOmega M=\wedge^{1,0}M=T^{*1,0}M$ and $\varOmega^p M:=\wedge^p\varOmega M=\wedge^{p,0}M$.

In a locally smooth unitary frame $e_1,\ldots,e_n$, the connection matrix $\theta=(\theta^i_j)$ of the Chern connection $\nabla$ is defined by
\[
    \nabla e_i = \theta^j_i \otimes e_j, \quad \theta^j_i~\text{are locally smooth 1-forms for all } i,j
\]
The \emph{Chern torsion} is 
\[
    T_{\nabla}(X,Y):=\nabla_XY-\nabla_YX-[X,Y],\quad \text{for any locally smooth vector fields } X,Y.
\]
Let $\{e^i\}$ be the dual frame of $\{e_i\}$ with respect to $h$, the $i$-th torsion form is 
\[
    \tau^i(X,Y) := e^i(T_{\nabla}(X,Y)),\quad \text{for any locally smooth vector fields } X,Y.
\]
and in tensor notation, write
\[
    \tau^i=\frac{1}{2} T^i_{jk} e^j\wedge e^k,\quad \tau = \tau^i\otimes e_i.
\]

\begin{proposition}\label{prop:basic_identities}
Let $(M^n,h)$ be a compact Chern--K\"ahler-like Hermitian manifold, and define $\dif^\nabla$ to be the covariant exterior derivative with respect to the Chern connection $\nabla$.
Then
\begin{align}
    \dif\omega^{n-1}&=0,
    \label{eq:balanced}\\
    \dif^\nabla\tau&=0,
    \label{eq:parallel_tau}\\
    \sqrt{-1}\,\partial\bar\partial\omega
    &=
    \sum_a\tau^a\wedge\overline{\tau^a}.
    \label{eq:wedge_of_tau}
\end{align}
In particular, $T_\nabla$ is a holomorphic tensor and
\begin{equation}
    \sum_a T^a_{aj}=0.
    \label{eq:trace_T}
\end{equation}

Moreover, define
\[
    \widehat R_\nabla:\Sym^2(T^{1,0}M)\longrightarrow\Sym^2(T^{1,0}M)
\]
by
\[
    h\bigl(\widehat R_\nabla(u\odot v),z\odot w\bigr) = R_\nabla(u,\bar z,v,\bar w),\quad u_1\odot\cdots\odot u_s :=\frac{1}{s!}\sum_{\sigma\in S_s}u_{\sigma(1)}\otimes\cdots\otimes u_{\sigma(s)},
\]
and define the torsion Bianchi map
\[
    \mathfrak B_\nabla: \wedge^3(T^{1,0}M) \longrightarrow \Sym^2(T^{1,0}M)
\]
by
\[
    \mathfrak B_\nabla(u\wedge v\wedge w) = T_\nabla(u,v)\odot w + T_\nabla(v,w)\odot u + T_\nabla(w,u)\odot v.
\]
Then
\begin{equation}\label{eq:bianchi_map}
    \widehat R_\nabla\circ\mathfrak B_\nabla=0.
\end{equation}
\end{proposition}

\begin{proof}
Let $\theta=(\theta^i_j)$ and $\Theta=(\Theta^i_j)$ be respectively the connection and curvature matrices of the Chern connection with respect to the unitary frame $\{e_i\}$.  The Chern structure equations are
\[
    \dif e^a = -\theta^a_b\wedge e^b+\tau^a, \qquad \Theta^a_b = \dif\theta^a_b+\theta^a_c\wedge\theta^c_b.
\]
Applying $\dif$ to the first structure equation gives the first Bianchi identity
\[
    \dif\tau^a = \Theta^a_b\wedge e^b-\theta^a_b\wedge\tau^b.
\]
Thus 
\[
    \dif^\nabla\tau =\dif\tau^a\otimes e_a+\tau^a\otimes\nabla e_a =\Theta^a_b\wedge e^b\otimes e_a =R^a_{bi\bar{j}} e^i\wedge\bar e^j\wedge e^b\otimes e_a 
\]
Since $h$ is Chern--K\"ahler-like,
\[
    R_{i\bar j k\bar\ell} = R_{k\bar j i\bar\ell} ~\Rightarrow R^a_{bi\bar j}=R^a_{ib\bar j},
\]
and hence
\[
    \Theta^a_b\wedge e^b=0,
\]
therefore
\[
    \dif^\nabla\tau=0,
\]
whose $(2,1)$-part is $\bar{\partial}^{\nabla}\tau=0$ which is equivalent to $T_\nabla$ being holomorphic.

The fundamental form is
\[
    \omega = \sqrt{-1}\sum_a e^a\wedge\bar{e}^a.
\]
Using the first structure equation and the fact that the Chern connection is Hermitian, one obtains
\[
    \partial\omega =\sqrt{-1}\sum_a \tau^a\wedge\bar{e}^a, \qquad\bar\partial\omega =-\sqrt{-1}\sum_a e^a\wedge\bar{\tau}^a.
\]
Applying $\dif$ to the second equality and using $\dif^\nabla\tau=0$ gives
\[
    \sqrt{-1}\,\partial\bar\partial\omega = \sum_a\tau^a\wedge\bar{\tau}^a.
\]

By the balancedness theorem for compact Chern--K\"ahler-like manifolds,
\[
    \dif\omega^{n-1}=0;
\]
see \cite[Theorem~1.3]{YangZheng2018CurvatureTensors}. The Lee form of $\omega$ is, up to a fixed nonzero normalization, the trace of the Chern torsion. Thus
\[
    \dif\omega^{n-1}=0 \quad\Longleftrightarrow\quad \sum_aT^a_{aj}=0
\]
for every $j$, proving \eqref{eq:trace_T}.

Finally, apply $\dif^\nabla$ to $\dif^\nabla\tau=0$. Since the square of the covariant exterior derivative is the curvature action, we obtain
\[
    0=(\dif^\nabla)^2\tau=\Theta^a_b\wedge\tau^b\otimes e_a.
\]
Write
\[
    \Theta^a_b=R^a_{b i\bar j}e^i\wedge\bar e^j,\qquad\tau^b=\frac12T^b_{k\ell}e^k\wedge e^\ell.
\]
The coefficient of
\[
    e^i\wedge e^k\wedge e^\ell\wedge\bar e^j\otimes e_a
\]
in the preceding identity is
\[
    \sum_b\left(R^a_{b i\bar j}T^b_{k\ell}+R^a_{b k\bar j}T^b_{\ell i}+R^a_{b\ell\bar j}T^b_{ik}\right)=0.
\]
After lowering the index $a$, this becomes
\[
    \sum_b\left(R_{i\bar j b\bar q}T^b_{k\ell}+R_{k\bar j b\bar q}T^b_{\ell i}+R_{\ell\bar j b\bar q}T^b_{ik}\right)=0.
\]
Using the CKL symmetry
\[
    R_{i\bar j b\bar q}=R_{b\bar j i\bar q},
\]
we obtain
\[
    \sum_b\left(T^b_{ik}R_{b\bar j\ell\bar q}+T^b_{k\ell}R_{b\bar j i\bar q}+T^b_{\ell i}R_{b\bar j k\bar q}\right)=0.
\]
Equivalently, for all $u,v,w,z,q\in T^{1,0}M$,
\[
\begin{aligned}
0=&R_\nabla\bigl(T_\nabla(u,v),\bar z,w,\bar q\bigr)\\
&+R_\nabla\bigl(T_\nabla(v,w),\bar z,u,\bar q\bigr)\\
&+R_\nabla\bigl(T_\nabla(w,u),\bar z,v,\bar q\bigr).
\end{aligned}
\]
By the definitions of $\widehat R_\nabla$ and $\mathfrak B_\nabla$, the right-hand side is
\[
    h\left((\widehat R_\nabla\circ\mathfrak B_\nabla)(u\wedge v\wedge w),z\odot q\right).
\]
Therefore
\[
    h\left((\widehat R_\nabla\circ\mathfrak B_\nabla)(u\wedge v\wedge w),z\odot q\right)=0
\]
for every $z,q\in T^{1,0}M$. Since the tensors $z\odot q$ span $\Sym^2(T^{1,0}M)$ and the induced Hermitian metric on $\Sym^2(T^{1,0}M)$ is nondegenerate, it follows that
\[
    (\widehat R_\nabla\circ\mathfrak B_\nabla)(u\wedge v\wedge w)=0.
\]
Since $u,v,w$ are arbitrary, we conclude that
\[
    \widehat R_\nabla\circ\mathfrak B_\nabla=0.
\]
\end{proof}

Set
\[
    \Sigma = \sum_a \tau^a\wedge \bar{\tau}^a,
\]
and define the Jacobiator coefficients
\[
    C_{ijk}^a=\sum_b(T_{ij}^bT_{bk}^a+T_{jk}^bT_{bi}^a+T_{ki}^bT_{bj}^a).
\]
For $n \ge 4$, we define the following top-degree form
\[
    \mathcal{Q}_n =  \begin{cases}
        \Sigma^2, & n=4,\\
        \Sigma^2\wedge \frac{\omega^{n-4}}{(n-4)!}+\sqrt{-1}\Sigma\wedge \partial \omega \wedge \bar{\partial} \omega\wedge \frac{\omega^{n-5}}{(n-5)!}, & n>4
    \end{cases}
\]
\begin{lemma}\label{lemma:Q_n}
    For $n\ge 4$, we have
    \[
        \mathcal{Q}_n=\sum_{a=1}^n\sum_{i<j<k}|C_{ijk}^a|^2\frac{\omega^n}{n!}.
    \]
    When $n=3$, the trace condition \eqref{eq:trace_T} implies that $C_{123}^a=0$ for every $a$.
\end{lemma}
\begin{proof}
The assertion is pointwise. Fix $x\in M$ and choose a local Chern-normal unitary frame $\{e_i\}_{i=1}^n$ at $x$ such that $\theta_b^a(x)=0$, with dual coframe $\{e^i\}_{i=1}^n$. Thus, at $x$,
\[
    \omega=\sqrt{-1}\sum_{r=1}^n e^r\wedge\bar e^r, \qquad \tau^a=\sum_{i<j}T^a_{ij}e^i\wedge e^j,
\]
and
\[
    \partial\omega=\sqrt{-1}\sum_b\tau^b\wedge\bar e^b,\qquad\bar\partial\omega=-\sqrt{-1}\sum_c e^c\wedge\bar\tau^c.
\]

Assume first that $n\geq4$. Define
\[
    W^{ab}:=\tau^a\wedge\tau^b\in\varOmega^4_x M,\qquad\mathcal J^a:=\sum_b\tau^b\wedge\iota_{b}\tau^a\in\varOmega^3_xM,
\]
where $\iota_b$ denotes the interior product with $e_b$. Thus
\[
    \iota_{b}\tau^a=\sum_qT^a_{bq}e^q,
\]
the coefficient of $e^i\wedge e^j\wedge e^k$, where $i<j<k$,
in $\mathcal J^a$ is
\[
    \sum_b\left(T^b_{ij}T^a_{bk}+T^b_{jk}T^a_{bi}+T^b_{ki}T^a_{bj}\right)=C^a_{ijk}.
\]
Hence
\begin{equation}\label{eq:Jacobiator}
    \mathcal J^a=\sum_{i<j<k}C^a_{ijk}e^i\wedge e^j\wedge e^k,\qquad\sum_a|\mathcal J^a|^2=\sum_a\sum_{i<j<k}|C^a_{ijk}|^2.
\end{equation}

Introduce the trace one-form
\[
    \vartheta:=\sum_b\iota_{b}\tau^b=\sum_j\left(\sum_bT^b_{bj}\right)e^j.
\]
By \eqref{eq:trace_T}, $\vartheta=0$. Since the $\tau^a$ have even degree, the contraction rule gives
\begin{equation}\label{eq: W_ab}
    \begin{aligned}
    \sum_b\iota_{e_b}W^{ab}&=\sum_b\iota_{e_b}(\tau^a\wedge\tau^b)\\
    &=\sum_b(\iota_{e_b}\tau^a)\wedge\tau^b+\tau^a\wedge\sum_b\iota_{e_b}\tau^b\\
    &=\sum_b\tau^b\wedge\iota_{e_b}\tau^a+\tau^a\wedge\vartheta=\mathcal J^a.
    \end{aligned}
\end{equation}

Let
\[
    \varepsilon_b(\alpha):=e^b\wedge\alpha.
\]
With respect to the Hermitian inner product induced by $h$ on forms, $\varepsilon_b^*=\iota_b$, and
\[
    \iota_c\varepsilon_b+\varepsilon_b\iota_c=\delta_{bc}\operatorname{Id}.
\]
Using \eqref{eq: W_ab}, adjointness, and the preceding anticommutation relation, we obtain
\begin{align}
    \sum_a|\mathcal J^a|^2 &=\sum_{a,b,c} \langle\iota_bW^{ab},\iota_cW^{ac}\rangle\notag\\
    &=\sum_{a,b,c} \langle W^{ab},\varepsilon_b\iota_cW^{ac}\rangle\notag\\
    &=\sum_{a,b}|W^{ab}|^2 -\sum_{a,b,c} \langle W^{ab},\iota_c\varepsilon_bW^{ac}\rangle\notag\\
    &=\sum_{a,b}|W^{ab}|^2-\sum_{a,b,c} \langle\varepsilon_cW^{ab},\varepsilon_bW^{ac}\rangle.\label{eq:W_ab_J}
\end{align}

For $\alpha,\beta\in\varOmega^r_xM$, the unitary Hodge
identity is
\begin{equation}\label{eq:Hodge_id}
    (\sqrt{-1})^{r^2}\alpha\wedge\bar\beta \wedge\frac{\omega^{n-r}}{(n-r)!}=\langle\alpha,\beta\rangle\frac{\omega^n}{n!}.
\end{equation}
Indeed, this follows by checking \eqref{eq:Hodge_id} on the orthonormal basis $\{e^{i_1}\wedge\cdots\wedge e^{i_r}\}_{i_1<\cdots<i_r}$.
Since
\[
    \Sigma^2=\sum_{a,b}W^{ab}\wedge\overline{W^{ab}},
\]
formula \eqref{eq:Hodge_id} with $r=4$ gives 
\begin{equation}\label{eq:Sigma_W_ab}
    \Sigma^2\wedge\frac{\omega^{n-4}}{(n-4)!}=\sum_{a,b}|W^{ab}|^2\frac{\omega^n}{n!}.
\end{equation}

Using the displayed formulas for $\partial\omega$ and $\bar\partial\omega$, and moving all holomorphic factors to the left, we get
\begin{align}
    &\sqrt{-1}\;\Sigma\wedge\partial\omega\wedge\bar\partial\omega\notag\\
    &\quad=-\sqrt{-1}\sum_{a,b,c}\varepsilon_cW^{ab}\wedge\overline{\varepsilon_bW^{ac}}.\label{eq:W_ab_wedge}
\end{align}
The minus sign in \eqref{eq:W_ab_wedge} is produced by moving $e^c$ past the three antiholomorphic one-forms occurring in $\bar\tau^a\wedge\bar e^b$. For $n>4$, applying \eqref{eq:Hodge_id} with $r=5$ to \eqref{eq:W_ab_wedge} yields
\begin{align}
    &\sqrt{-1}\,\Sigma\wedge\partial\omega\wedge\bar\partial\omega\wedge\frac{\omega^{n-5}}{(n-5)!}\notag\\
    &\quad=-\sum_{a,b,c}\langle\varepsilon_cW^{ab},\varepsilon_bW^{ac}\rangle\frac{\omega^n}{n!}.\label{eq:Sigma_W_ab_inner}
\end{align}
Combining \eqref{eq:Jacobiator}, \eqref{eq:W_ab_J}, \eqref{eq:Sigma_W_ab}, and \eqref{eq:Sigma_W_ab_inner}, we conclude that, for $n>4$,
\[
    \mathcal Q_n=\sum_a|\mathcal J^a|^2\frac{\omega^n}{n!}=\sum_{a=1}^n\sum_{i<j<k}|C^a_{ijk}|^2\frac{\omega^n}{n!}.
\]

If $n=4$, then $\varepsilon_cW^{ab}\in\varOmega^5_xM$. Thus the second sum on the right-hand side of \eqref{eq:W_ab_J} vanishes, and \eqref{eq:Jacobiator} and \eqref{eq:Sigma_W_ab} give
\[
    \mathcal Q_4=\Sigma^2=\sum_{a=1}^4\sum_{i<j<k}|C^a_{ijk}|^2\frac{\omega^4}{4!}.
\]

It remains to treat $n=3$. There is only one increasing triple. Expanding the definition of $C^a_{123}$ and using $T^a_{ij}=-T^a_{ji}$, we obtain
\[
\begin{aligned}
    C^a_{123}=&(T^1_{12}-T^3_{23})T^a_{13}-(T^1_{13}+T^2_{23})T^a_{12}\\
            &+(T^2_{12}+T^3_{13})T^a_{23}.
\end{aligned}
\]
On the other hand, \eqref{eq:trace_T}, for $j=1,2,3$, gives
\[
    T^2_{12}+T^3_{13}=0, \qquad T^1_{12}-T^3_{23}=0, \qquad T^1_{13}+T^2_{23}=0.
\]
All three coefficients in the preceding expression vanish. Therefore $C^a_{123}=0$ for every $a$.
\end{proof}

\begin{theorem}
    Let $(M,h)$ be a compact CKL manifold. Then, for every $x\in M$, the Chern torsion algebra $\mathfrak A_{x,\nabla}$ is a complex Lie algebra.
\end{theorem}
\begin{proof}
Let $\{e_i\}_{i=1}^n$ be a unitary basis of $T^{1,0}_xM$ and write
\[
    T_\nabla(e_i,e_j)=\sum_aT^a_{ij}e_a.
\]
The Jacobiator is
\[
\begin{aligned}
    J_\nabla(e_i,e_j,e_k):=&[[e_i,e_j]_\nabla,e_k]_\nabla+[[e_j,e_k]_\nabla,e_i]_\nabla\\
        &+[[e_k,e_i]_\nabla,e_j]_\nabla .
\end{aligned}
\]
A direct calculation gives
\[
    J_\nabla(e_i,e_j,e_k) = \sum_a C^a_{ijk}e_a,
\]
where
\[
    C^a_{ijk} = \sum_b\left(T^b_{ij}T^a_{bk}+T^b_{jk}T^a_{bi}+T^b_{ki}T^a_{bj}\right).
\]
Thus it suffices to prove
\[
    C^a_{ijk}=0
\]
for all $a,i,j,k$.

If $n\leq2$, then
\[
    \wedge^3T^{1,0}_xM=0,
\]
so every skew-symmetric complex bilinear bracket automatically satisfies the Jacobi identity.

Suppose next that $n=3$.  Since every compact CKL metric is balanced, its Chern torsion is trace-free:
\[
    \sum_aT^a_{aj}=0.
\]
By Lemma~\ref{lemma:Q_n}, this implies
\[
    C^a_{123}=0
\]
for every $a$.  By skew-symmetry, these are all the possibly nonzero components of the Jacobiator.  Hence the Jacobi identity holds when $n=3$.

It remains to consider $n\geq4$.  Set
\[
    \Sigma=\sum_a\tau^a\wedge\overline{\tau^a}.
\]
By the basic CKL identities,
\[
    \Sigma = \sqrt{-1}\,\partial\bar\partial\omega.
\]
In particular,
\[
    \partial\Sigma=0, \qquad \bar\partial\Sigma=0.
\]

When $n=4$, we have
\[
    \mathcal Q_4=\Sigma^2.
\]
Since $\Sigma=\sqrt{-1}\partial\bar\partial\omega$ and $d\Sigma=0$,
\[
\begin{aligned}
    \mathcal Q_4&=\sqrt{-1}\partial\bar\partial\omega\wedge\Sigma\\
                &=\sqrt{-1}\partial\bar\partial(\omega\wedge\Sigma).
\end{aligned}
\]
Therefore \(\mathcal Q_4\) is $\partial\bar\partial$-exact.  By Stokes' theorem,
\[
    \int_M\mathcal Q_4=0.
\]

Now suppose that $n>4$.  Since $\partial\Sigma=\bar\partial\Sigma=0$, we have
\[
    \sqrt{-1}\partial\bar\partial\left(\Sigma\wedge\frac{\omega^{n-3}}{(n-3)!}\right)=\sqrt{-1}\Sigma\wedge\frac{\partial\bar\partial\omega^{n-3}}{(n-3)!}.
\]
For every positive integer $m$,
\[
    \partial\bar\partial\omega^m=m\,\partial\bar\partial\omega\wedge\omega^{m-1}+m(m-1)\,\partial\omega\wedge\bar\partial\omega\wedge\omega^{m-2}.
\]
Taking $m=n-3$ gives
\[
\begin{aligned}
    &\sqrt{-1}\partial\bar\partial\left(\Sigma\wedge\frac{\omega^{n-3}}{(n-3)!}\right)\\
    &\quad=\sqrt{-1}\Sigma\wedge\partial\bar\partial\omega\wedge\frac{\omega^{n-4}}{(n-4)!}\\
    &\qquad+\sqrt{-1}\Sigma\wedge\partial\omega\wedge\bar\partial\omega\wedge\frac{\omega^{n-5}}{(n-5)!}.
\end{aligned}
\]
Using
\[
    \sqrt{-1}\partial\bar\partial\omega=\Sigma,
\]
we obtain
\[
\begin{aligned}
    \sqrt{-1}\partial\bar\partial\left(\Sigma\wedge\frac{\omega^{n-3}}{(n-3)!}\right)=&\Sigma^2\wedge\frac{\omega^{n-4}}{(n-4)!}\\
    &+\sqrt{-1}\Sigma\wedge\partial\omega\wedge\bar\partial\omega\wedge\frac{\omega^{n-5}}{(n-5)!}\\
    =&\mathcal Q_n.
\end{aligned}
\]
Thus $\mathcal Q_n$ is $\partial\bar\partial$-exact for every $n>4$.  Since $M$ is compact and has no boundary, Stokes'
theorem yields
\[
    \int_M\mathcal Q_n=0.
\]

On the other hand, Lemma~\ref{lemma:Q_n} gives
\[
    \mathcal Q_n =\sum_{a=1}^n\sum_{i<j<k} |C^a_{ijk}|^2\frac{\omega^n}{n!}.
\]
Consequently,
\[
    0 = \int_M\mathcal Q_n = \int_M \sum_{a=1}^n\sum_{i<j<k} |C^a_{ijk}|^2\frac{\omega^n}{n!}.
\]
The integrand is a smooth nonnegative function.  It must therefore vanish identically:
\[
    C^a_{ijk}=0
\]
for every $a,i,j,k$ and at every point of $M$.

Hence
\[
    J_\nabla(u,v,w)=0
\]
for all $u,v,w\in T^{1,0}_xM$.  The bracket
\[
    [u,v]_\nabla=T_\nabla(u,v)
\]
is complex bilinear, skew-symmetric, and satisfies the Jacobi identity.  Therefore $\mathfrak A_{x,\nabla}$ is a complex Lie algebra.
\end{proof}

\section{Pure Pl\"{u}cker power vanishing}

	For a $(p,0)$-form $\eta$, we say it is \emph{pointwise decomposable} if it is a wedge of $p$ covectors in $\varOmega_xM$ for each $x$ and is \emph{decomposable} if it is $p\,$-th wedge product of $(1,0)$-forms.
	
	We define
	\[
	\mathcal{T}\eta := \sum_a \tau_a \wedge \iota_{a}\eta,
	\]
	where $\iota_{a}$ is the interior product with the vector $e_a$.
	And there is semipositive $(1,1)$-form and $(n-1,n-1)$-form
	\[
	\beta_\eta:=\Lambda_\omega^{p-1}\left(\frac{(\sqrt{-1})^{p^2}}{p!}\eta\wedge\bar{\eta}\right),\quad \Psi_\eta:=(\sqrt{-1})^{p^2}\eta\wedge\bar{\eta}\wedge\frac{\omega^{n-p-1}}{(n-p-1)!},
	\]
	where $\Lambda_\omega$ is the adjoint of Lefschetz operator $L_\omega (\cdot):=\omega\wedge\cdot$ with respect to the Hermitian metric on the space of forms, which we denote by $\langle\alpha,\beta\rangle$ such that $\langle\alpha,\beta\rangle\frac{\omega^n}{n!}=\alpha\wedge\ast\bar\beta$.
	Note that \[\iota_{a}\eta=-\sqrt{-1}\Lambda_\omega(\bar{e}^a\wedge\eta),\quad\forall~ (p, 0) \text{-form }\eta,\] we have
	\[\Lambda_\omega(\partial\omega\wedge\eta)=-\sum_{a}\iota_{a}(\tau^a\wedge\eta)=-\mathcal{T}\eta, \]since $\tau$ is traceless. This formula also implies that $\mathcal{T}\eta$ is a global $(p+1)\,$-form and $\mathcal{T}\eta=0$ when $p\ge n-1$.
	
\begin{lemma}\label{prop:decomp-square}
		Let $(M,h)$ be a compact Chern--K\"ahler-like manifold, if $\eta\in H^0\!\left(M,\varOmega^pM\right)$, i.e. $\eta$ is a holomorphic $(p,0)$-form, $1\le p\le n-1$, then
		\[
		\sqrt{-1}\partial\bar{\partial}\Psi_\eta=
		\sqrt{-1}\partial\bar{\partial}\left( \sqrt{-1}^{p^2}\eta\wedge\bar{\eta}\wedge\frac{\omega^{n-p-1}}{(n-p-1)!}\right)=|\partial\eta-\mathcal{T}\eta|^2\frac{\omega^n}{n!}. 
		\]
	\end{lemma}
	\begin{proof}
		For any $(p,q)$ form $\xi$ satisfying $\Lambda^2_\omega\xi=0, p+q\le n$ , we have the Lefschetz decomposition \[\xi=\left(\xi-\frac{L_\omega\Lambda_\omega\xi}{n-p-q+2}\right)+\frac{\Lambda_\omega\xi}{n-p-q+2}\wedge\omega\]
		and the Hodge-star calculation gives 
		\[\ast\xi=\sqrt{-1}^{(p+q)^2+2p}\left(\frac{\omega^{n-p-q}\wedge\xi}{(n-p-q)!}-\frac{\omega^{n-p-q+1}\wedge\Lambda_\omega\xi}{(n-p-q+1)!}\right).\]
		If we make the convention that $\frac{\omega^{n-p-q}}{(n-p-q)!}=0$ when $n-p-q<0$, then the equality holds when $p+q=n+1$, since there is no primitive $(n+1)$ form and $\xi=L_\omega\Lambda_\omega\xi$ under this condition.
		Thus there are calculations
		
		\begin{align}
			|\partial\eta|^2\frac{\omega^n}{n!}&=\sqrt{-1}^{(p+1)^2}\frac{\partial\eta\wedge\overline{\partial\eta}\wedge\omega^{n-p-1}}{(n-p-1)!}, \\
			-\langle\mathcal{T\eta},\partial\eta\rangle\frac{\omega^n}{n!}&=\langle\partial\omega\wedge\eta,\omega\wedge\partial\eta\rangle\frac{\omega^n}{n!}=\sqrt{-1}^{(p+1)^2}\partial\omega\wedge\eta\wedge\frac{\overline{\partial\eta}\wedge\omega^{n-p-2}}{(n-p-2)!}, \\
			|\partial\omega\wedge\eta|^2\frac{\omega^n}{n!}&=\sqrt{-1}^{(p+3)^2+2}\partial\omega\wedge\eta\wedge\left(\frac{\bar\partial\omega\wedge\bar\eta\wedge\omega^{n-p-3}}{(n-p-3)!}-\frac{\overline{\mathcal{T}\eta}\wedge\omega^{n-p-2}}{(n-p-2)!}\right) \notag \\
			&=-\sqrt{-1}^{(p+1)^2}\frac{\partial\omega\wedge\eta\wedge\bar\partial\omega\wedge\bar\eta\wedge\omega^{n-p-3}}{(n-p-3)!}+|\mathcal{T}\eta|^2\frac{\omega^n}{n!}.
		\end{align}
		Note that when $p\ge n-1$, $\partial\omega\wedge\eta=0$ and $\mathcal{T}\eta=0$; when $p=n-2$, the corresponding term of $\frac{\omega^{n-p-3}}{(n-p-3)!}$ is interpreted as 0 by the convention above and is absent in $\partial\bar{\partial}\Psi_\eta$.
		Meanwhile, we have
		\begin{equation}
			\begin{aligned}
			|\partial\omega\wedge\eta|^2\frac{\omega^n}{n!}&=|\Sigma_{a}\tau^a\wedge\overline{e^a}\wedge\eta|^2\frac{\omega^n}{n!}=\sum_{a}|\tau^a\wedge\eta|^2\frac{\omega^n}{n!}\\
			&=\sqrt{-1}^{(p+2)^2}\sum_{a}\tau^a\wedge\eta\wedge\frac{\overline{\tau^a}\wedge\bar\eta\wedge\omega^{n-p-2}}{(n-p-2)!} \\
			&=\sqrt{-1}^{p^2+1}\eta\wedge\bar\eta\wedge\frac{\partial\bar\partial\omega\wedge\omega^{n-p-2}}{(n-p-2)!}.
			\end{aligned}
		\end{equation}
		Combining the four equations above, we get
		\[
		\partial\bar{\partial}\left(\sqrt{-1}^{p^2+1}\eta\wedge\bar{\eta}\wedge\frac{\omega^{n-p-1}}{(n-p-1)!}\right)=
		|\partial\eta|^2\frac{\omega^n}{n!}-2\text{Re}\langle\mathcal{T\eta},\partial\eta\rangle\frac{\omega^n}{n!}+|\mathcal{T}\eta|^2\frac{\omega^n}{n!}
		= |\partial\eta-\mathcal{T}\eta|^2\frac{\omega^n}{n!}.
		\]
    \end{proof}

	For a $(p,0)$-form $\eta$, $\beta_\eta$ is a semipositive $(1,1)$-form, if its eigenvalues are $b_i\ge0$, we have
	\begin{equation}\label{eq:bochner-beta}
		\langle\sqrt{-1}\partial\bar\partial|\eta^2|,\beta_\eta\rangle
		=\sum_i b_i|\nabla'_{e_i}\eta|^2+p\sum_{i,j}R_{i\bar i j\bar j}b_i b_j
		=\sum_i b_i|\nabla'_{e_i}\eta|^2+p\text{RBC}_h(\{e_i\},\{b_i\})
	\end{equation}
	where $\{e_i\}$ is the corresponding eigenbasis. It is a corollary of the Bochner-type formula used in \cite[Lemma~2.2]{Tang2026}.
	
	Also, use a similar calculation as the lemma above, we have for all $(1,1)$ form $\alpha$,
	\begin{equation}\label{eq:hodge-1}
		\sqrt{-1}^{p^2}\alpha\wedge\eta\wedge\bar\eta\frac{\omega^{n-p-1}}{(n-p-1)!}=\left(\tr_\omega\alpha\cdot|\eta|^2-p\langle\alpha,\beta_\eta\rangle\right)\frac{\omega^n}{n!}.
	\end{equation}

\begin{theorem}\label{thm:plucker_vanishing}
	Let $(M,h)$ be a compact Chern--K\"ahler-like manifold,	with strongly quasi-positive real bisectional curvature, and let
	\[
	A\in H^0\!\left(M,\Sym^m(\varOmega^pM)\right),
	\qquad 1\le p\le n,\quad m\ge1.
	\]
	Suppose that there is a dense open set $U$ such that $\forall x\in U$, 
	\[
	A_x=\left(\alpha_1\wedge\cdots\wedge\alpha_p\right)^{\odot m},\quad\exists\alpha_1,\cdots,\alpha_p\in \varOmega_xM.
	\]
	Then $A=0$.
\end{theorem}
\begin{proof}
	We divide the proof into five explicit steps.
	
	\smallskip
	\noindent\textbf{Step 1: the case $p=n$.}
	Then $A$ is a section of canonical bundle $K_M^m$.  The Chern scalar curvature is
	\[
	s_C=\sum_{i,j}R_{i\bar i j\bar j}=\RBC_h(e,(1,\dots,1)).
	\]
	It is nonnegative everywhere and positive on a nonempty open set.  Since
	$\omega$ is balanced, it is Gauduchon.  The integrated Bochner formula for
	$0\ne A\in H^0(M,K_M^m)$ is
	\[
	0=\int_M\left(|{\nabla'A}|^2+m\cdot s_C|A|^2\right)\frac{\omega^n}{n!}.
	\]
	Both summands are nonnegative.  A nonzero holomorphic section cannot vanish on an open set, so the second summand is positive somewhere, a contradiction.  Hence $A=0$.
	
	\smallskip
	For the rest of the proof, assume $p\le n-1$ and put $E=\varOmega^pM$.
	
	\smallskip
	\noindent\textbf{Step 2: construction of a compact root space.}
	Composing the Pl\"{u}cker embedding and the Veronese map pointwise gives us an injective bundle map 
	\[\mathrm{Gr}(p,\varOmega M)\hookrightarrow\mathbb{P}(E)\hookrightarrow\mathbb{P}(\Sym^mE).\]
	Its image is closed in $\mathbb{P}(\Sym^mE)$ and the corresponding pointwise affine cone $\mathcal{C}\subset\Sym^mE$ is also closed. By continuity, $A\in\mathcal{C}$. Thus on each point $x\in M$ where $A_x\neq0$, we can find a unique line $[\eta_x]\in\mathbb{P}(E_x)$ such that $[\eta_x^{\odot m}]=[A_x]\in\mathbb{P}(\Sym^mE)$. The Veronese map is an embedding; by viewing $\mathbb{P}(E)$ as a subvariety of $\mathbb{P}(\Sym^mE)$ and imposing the definition equations of $[A]$ (and probably their derivations), the closure of the graph of $x\mapsto[\eta_x]$ is a subvariety of $\mathbb{P}(\Sym^mE)$ (thus this is a meromorphic map). Resolving the closure of the graph, we obtain a compact connected complex manifold $X$, a modification $\mu:X\to M$ and a holomorphic map $\gamma:X\to\mathbb{P}(E)$ which is the `extension' of $x\mapsto[\eta_x]$ on $X$.
	Put \[L=\gamma^*\mathcal{O}_{\mathbb{P}(E)}(-1)\subset\mu^*E,\] and the section $\mu^*A$ takes values in $L^{\otimes m}$ on a dense open set and therefore everywhere, denote it by $s\in H^0(X,L^{\otimes m})$. The projection $\pi:L\to X$ induces a tautological section $t\in\pi^*L$, and the analytic subspace $Z=\{l\in L~|~t(l)^{\otimes m}=\pi^*s(l)\}$ is finite and flat over $X$, since locally it is defined by the monic equation $t^m-s=0$. Normalize $Z$, choose a connected component dominating $X$, and resolve its singularities, we obtain a compact connected complex manifold $Y$, a holomorphic map $\varphi:Y\to Z$ and a generically finite surjective holomorphic map $F:Y\to M$ (by composing all the modifications). The whole procession is illustrated by
		\[\begin{tikzcd}
		&                   & \pi^*L \arrow[d]   & L \arrow[d]        & E \arrow[d] \\
		Y \arrow[r, "\varphi"] & Z \arrow[r, hook] & L \arrow[r, "\pi"] & X \arrow[r, "\mu"] & M          
		\end{tikzcd}\]
	And since $L\subset\mu^*E$, we have
	\[\varphi^*t_{\restriction_Z}\in H^0(Y,F^*E),\quad (\varphi^*t_{\restriction_Z})^{\otimes m}=F^*A.\] 
	For normalization and desingularization, see \cite{Grauert1984} and \cite{BierstoneMilman1997}, and $F:Y\to M$ is surjective and is a covering map except for some ramification points and exceptional loci.
	So there is a natural holomorphic map $dF^{*\wedge p}:F^*E\to\varOmega^pY$, using which we define a genuine holomorphic form \[\widetilde\eta=(dF^{*\wedge p})(\varphi^*t_{\restriction_Z})\in H^0(Y,\varOmega^pY).\] Let $Y^\circ\subset Y$ be the dense open set on which $F$ is a local biholomorphism and $A\circ F\ne0$, where $dF^{*\wedge p}$ identifies $F^*E$ with $\varOmega^pY$, and $\widetilde{\eta}\neq0$. 
	
	\smallskip
	\noindent\textbf{Step 3: smooth global forms.}
	Set
	\[
	\Omega:=F^*\omega,
	\qquad f:=|\varphi^*t_{\restriction_Z}|_{F^*h}^2,
	\qquad
	\Psi:=\sqrt{-1}^{p^2}\widetilde\eta\wedge\overline{\widetilde\eta}
	\wedge\frac{\Omega^{n-p-1}}{(n-p-1)!}.
	\]
	These are smooth on all of $Y$. On $Y^\circ$, the metric $\Omega$ is positive and $f=|\widetilde{\eta}|_\Omega^2$; $\Omega$ may be semipositive rather than positive along $Y\backslash Y^\circ$, but no inverse metric is used in the global integration below. Define
	\begin{align}
		E_f&:=\sqrt{-1}\partial f\wedge\bar\partial f
		\wedge\frac{\Omega^{n-1}}{(n-1)!},\label{eq:Ef}\\
		B_{\widetilde\eta}&:=
		f\,\sqrt{-1}\partial\bar\partial f\wedge\frac{\Omega^{n-1}}{(n-1)!}
		-\sqrt{-1}\partial\bar\partial f\wedge\Psi,\label{eq:Betaform}\\
		S_{\widetilde\eta}&:=\sqrt{-1}\partial\bar\partial\Psi.\label{eq:Seta}
	\end{align}
	Set $\beta_{\widetilde{\eta}}=\Lambda^{p-1}_\Omega(\frac{\sqrt{-1}^{p^2}}{p!}\widetilde{\eta}\wedge\overline{\widetilde\eta})$, and all the pointwise calculations above still holds true on $Y^\circ$ when we substitute $\omega$ with $\Omega$. Therefore, \eqref{eq:hodge-1} and Lemma \ref{prop:decomp-square} give
	\begin{align}
		E_f&=|{\partial f}|^2\frac{\Omega^n}{n!},\label{eq:Ef-positive}\\
		B_{\widetilde\eta}&=
		p\langle\sqrt{-1}\partial\bar\partial f,\beta_{\widetilde\eta}\rangle
		\frac{\Omega^n}{n!},\label{eq:B-positive}\\
		S_{\widetilde\eta}&=
		|\partial\widetilde\eta-\mathcal{T}\widetilde\eta|^2
		\frac{\Omega^n}{n!}.
		\label{eq:S-positive}
	\end{align}
	
	On $Y^\circ$, $\Omega$ is a positive CKL metric, and by \eqref{eq:bochner-beta} and nonnegative RBC, all three top forms are nonnegative on $Y^\circ$. 
	
	\smallskip
	\noindent\textbf{Step 4: the compact Stokes identity.}
	Since $\omega$ is balanced,
	\[\dif\Omega^{n-1}=F^*(\dif\omega^{n-1})=0.\]
	By the Stokes identity,
	\[
	\int_Y f\,\sqrt{-1}\partial\bar\partial f
	\wedge\frac{\Omega^{n-1}}{(n-1)!}
	=-\int_Y\sqrt{-1}\partial f\wedge\bar\partial f
	\wedge\frac{\Omega^{n-1}}{(n-1)!}
	=-\int_YE_f,
	\]
	\[
	\int_Y\sqrt{-1}\partial\bar\partial f\wedge\Psi
	=\int_Y f\,\sqrt{-1}\partial\bar\partial\Psi
	=\int_YfS_{\widetilde\eta}.
	\]
	Substitution into \eqref{eq:Betaform} yields
	\begin{equation}\label{eq:stokes-master}
		0=\int_YE_f+\int_YB_{\widetilde\eta}+\int_YfS_{\widetilde\eta}.
	\end{equation}
	Every form on the right is nonnegative on $Y^\circ$, thus all three integrals vanish. On $Y^\circ$, \eqref{eq:Ef-positive} gives $\dif{f}=0$. The complement of a proper analytic subset in a connected complex manifold is connected, so $f$ is constant on $Y^\circ$, hence on $Y$ by continuity. If $A\ne0$, this constant is positive, thus 
	\[f\equiv C>0, \quad\widetilde\eta\neq0, \quad fS_{\widetilde\eta}=0\Rightarrow S_{\widetilde\eta}=0.\]
	
	\smallskip
	\noindent\textbf{Step 5: contradiction at the strong point.}
	Strict RBC positivity at the strong point persists on a neighborhood, this neighborhood intersects with the dense set $F(Y^{\circ})\subset M$ since $F$ is surjective. Choose $y\in Y^\circ$ inside the intersection. Since $\widetilde{\eta}\neq0$, the eigenvalues $b_i$ of $\beta_{\widetilde\eta}$ are nonnegative and not all zero. Formula \eqref{eq:bochner-beta} is then strictly positive at $y$, so by \eqref{eq:B-positive},
	$B_{\widetilde\eta}(y)>0$.  This contradicts the already proved identity
	$B_{\widetilde\eta}\equiv0$.  Hence $A\equiv0$.
	
\end{proof}

\begin{corollary}[Vanishing of holomorphic symmetric $2$-forms]\label{cor:sym2-vanish}
Let $(M, h)$ be a compact Chern--K\"{a}hler-like manifold. Under the assumption of strongly quasi-positive real bisectional curvature,
\[
 H^0(M,\Sym^2\varOmega M)=0.
\]
Consequently, if $L_u(v):=T_\nabla(u,v)$, then
\begin{equation}\label{eq:killing-zero}
 \kappa_\nabla(u,v):=\tr(L_uL_v)=0.
\end{equation}
\end{corollary}

	\begin{proof}
		Let $q\in H^0(M,\Sym^2\varOmega M)$ and let $r$ be its generic rank.  On the dense open rank-$r$ locus, view $q$ as a symmetric linear map
		$T^{1,0}M\to T^{*1,0}M$.  Over $\mathbb C$, a symmetric bilinear form is diagonalizable by congruence, so locally
		\[
		q=k_1\alpha_1^{\odot2}+\cdots+k_r\alpha_r^{\odot2},
		\qquad k_i\ne0.
		\]
		Define $Q\in H^0\!\left(M,\Sym^2(\varOmega^rM)\right)$ by taking the top nonzero compound, locally
		\[
		Q=(k_1\cdots k_r)
		(\alpha_1\wedge\cdots\wedge\alpha_r)^{\odot2}.
		\]
		Thus $Q$ is a pointwise pure square of a decomposable $r$-form.  By
		Theorem~\ref{thm:plucker_vanishing}, $Q=0$, contradicting the definition of $r$ unless $r=0$.  Hence $q=0$.
		
		The torsion $T_\nabla$ is holomorphic, so $u,v\mapsto\tr(L_uL_v)$ is a holomorphic covariant $2$-tensor.  Cyclicity of trace gives
		$\tr(L_uL_v)=\tr(L_vL_u)$, hence it lies in
		$H^0(M,\Sym^2\varOmega M)$ and must vanish.
        
	\end{proof}

\section{Solvability and Nilpotency of Chern torsion algebra}
At each point $x\in M$, we now obtain a finite-dimensional complex Lie algebra $\mathfrak{A}_{x,\nabla}$. By Corollary~\ref{cor:sym2-vanish} the Killing form $\kappa_\nabla(u,v)=\tr(\ad_u\ad_v)$ vanishes and then $\mathfrak{A}_{x,\nabla}$ is solvable by Cartan's criterion. By Lie's theorem we can simultaneously upper triangularize the adjoint representation. The diagonal elements of the adjoint representation are characters of this Lie algebra (with multiplicity):
\[
    \lambda_1,\ldots,\lambda_n \in \mathfrak{A}_{x,\nabla}^*.
\]

\begin{proposition}\label{prop:vanish_characters}
    If $M$ is a compact connected CKL manifold with strongly quasi-positive holomorphic sectional curvature (or real bisectional curvature), then every adjoint character $\lambda_i$ vanishes on $M$, i.e. $\ad_u$ is nilpotent for any $u\in\mathfrak{A}_{x,\nabla}$ at each $x\in M$.
\end{proposition}
\begin{proof}

For each $u\in\mathfrak{A}_{x,\nabla}$, the adjoint representation
\[
    \ad: \mathfrak{A}_{x,\nabla}\to \End(\mathfrak{A}_{x,\nabla}),\quad \ad_u(v):=[u,v]_\nabla
\]
is holomorphically dependent on $x$. Since $\mathfrak{A}_{x,\nabla}$ is solvable, then we define
\[
    P_s(u_1,\ldots,u_s) = \frac{1}{s!}\sum_{\sigma\in S_s} \tr(\ad_{u_{\sigma(1)}}\cdots\ad_{u_\sigma(s)}).
\]
Since $\ad_u$ is holomorphic for each $u$, then $P_s$ is a holomorphic symmetric $s$-tensor.

Fix some $x\in M$. By Lie's Theorem, we can choose some basis of $\mathfrak{A}_{x,\nabla}$ to simultaneously upper-triangularize every $\ad_u$. Suppose the diagonal elements are
\[
    \lambda_1,\ldots,\lambda_n \in \mathfrak{A}_{x,\nabla}^*.
\]
These are characters of adjoint representations. Then
\[
    P_s(u_1,\ldots,u_s)_x = \sum_{\nu=1}^n\lambda_\nu(u_1)\cdots\lambda_\nu(u_s).
\]
For $r,k\ge 1$, we define
\[
    W_{r,k} := \sum_{\substack{I\subset\{1,\ldots,n\}\\|I|=r}}\left(\bigwedge_{i\in I}\lambda_i\right)^{\odot 2k}.
\]
Although each $\lambda_i$ cannot be chosen holomorphically, we show that $W_{r,k}$ can be extended as a globally holomorphic tensor. 

\textbf{Step 1.} $W_{r,k}$ can be extended as a globally holomorphic tensor.

For any $r,q$, we define
\[
     \mathcal{A}_{r,q}:(\Sym^q \varOmega M)^{\otimes r}\to \Sym^q(\varOmega^rM),\quad \xi_1^{\odot q} \otimes\cdots\otimes \xi_r^{\odot q}\mapsto (\xi_1\wedge\cdots\wedge\xi_r)^{\odot q},
\]
which is a holomorphic vector bundle morphism.

For each point $x$, we have
\[
    P_{2k}|_x=\sum_{i=1}^n \lambda_i(x)^{\odot 2k}.
\]
Thus,
\[
    \mathcal{A}_{r,2k}(P_{2k}^{\odot r})|_x = \sum_{a_1,\ldots,a_r} (\lambda_{a_1}\wedge\cdots \wedge\lambda_{a_r})^{\odot 2k}.
\]
Since any permutation on $a_1,\ldots,a_r$ does not change the summation, we have
\[
    \mathcal{A}_{r,2k}(P_{2k}^{\odot r})|_x = r! \sum_{\substack{I\subset\{1,\ldots,n\}\\|I|=r}}\left(\bigwedge_{i\in I}\lambda_i(x)\right)^{\odot 2k}.
\]
Thus, globally we have
\[
W_{r,k}=\frac{1}{r!}\mathcal{A}_{r,2k}(P_{2k}^{\odot r}),
\]
which is a globally holomorphic tensor.

Set 
    \[
        L_\lambda(x)=\Span_{\mathbb{C}} \{\lambda_i(x):1\le i\le n\},\quad R := \max_M\dim_{\mathbb{C}} L_\lambda(x).
    \]
\textbf{Step 2.} Suppose that $R>0$ and we prove that there is some $k_0$ such that $W_{R,k_0}$ is not always zero and is a symmetric power of some decomposable form on $U$.

Take some point $x_0$ such that $\dim_\mathbb{C}L_\lambda(x_0)=R$. And choose a nonzero $\eta_{x_0}\in\Lambda^RL_\lambda(x_0)$, then for any $R$-element set $I$ there is some $c_I\in\mathbb{C}$ such that
\[
    \bigwedge_{i\in I}\lambda_i(x_0) = c_I\eta_{x_0},
\]
where at least one $c_I\ne 0$. Let $d_I=c_I^2$ and $N=\binom{n}{R}$. Since elementary symmetric polynomials of $d_I$'s cannot be all zeros, then power sum of $d_I$'s cannot be all zeros, that is there is some $k_0$ such that $W_{R,k_0}(x_0)\ne 0$. 
Let
\[
    U=\{x\in M : W_{R,k_0}\ne 0\}.
\]
Since $W_{R,k_0}$ is a holomorphic tensor, then $U$ is dense open. And pointwisely $W_{R,k_0}$ is a symmetric power of decomposable $\eta_{x_0}$.

From \textbf{Step 1 and 2} we prove that $W_{R,k_0}$ is a symmetric power of some pointwise decomposable form on $U$. And it is not always zero on $M$.
 
But by Proposition~\ref{prop:equiv_HSC_RBC} and Theorem~\ref{thm:plucker_vanishing}, $W_{R,k_0}\equiv0$. Thus, $R=0$ for any points and all $\lambda_i$ vanish.
\end{proof}

\begin{remark}
    In fact, $W_{R,k_0}$ is a symmetric power of decomposable $\eta$ on some dense open subset of $M$. Define a holomorphic map 
\[
    \Phi_W:U\to \mathbb{P}(\Sym^{2k_0}(\varOmega^RM))\quad x\mapsto [W(x)].
\]
Let
\[
    \mathrm{Pl}:\mathrm{Gr}(R,\varOmega M)\to\mathbb{P}(\varOmega^R M) 
\]
and 
\[
    \nu_{2k_0}: \mathbb{P}(\varOmega^R M)\to \mathbb{P}(\Sym^{2k_0}(\varOmega^R M)),\quad [\eta]\mapsto [\eta^{\odot 2k_0}]
\]
be Pl\"{u}cker and Veronese embeddings. Set $\Psi = \nu_{2k_0} \circ\mathrm{Pl}$ which is a closed holomorphic embedding. Since pointwisely $W_{R,k_0}(x)$ is a symmetric power of some decomposable wedge product, then $[W_{R,k_0}(x)]$ is in the image of $\Psi$. On its image, $\Psi^{-1}$ is also holomorphic. Thus, 
\[
    x\mapsto \Psi^{-1}([W_{R,k_0}(x)]) 
\]
is a holomorphic map to Grassmannian bundle. It determines a holomorphic map
\[
    x\mapsto [\eta_x]\in \mathbb{P}(\varOmega^R M) 
\]
again by Pl\"{u}cker embedding. This map gives a holomorphic line bundle
\[
    \mathcal{L}\subset \varOmega^R M|_U.
\]
For any $x\in U$, choose a small enough simply connected coordinate chart $V\subset U$ such that $\mathcal{L}|_V$ is trivial. Choose a holomorphic frame $\theta$ of $\mathcal{L}|_V$. Since each fiber of $\mathcal{L}$ is a Pl\"{u}cker line, $\theta$ is decomposable. Since $W$ and $\theta^{\odot 2k_0}$ generate the same line pointwisely, then there is a unique holomorphic function $f$ such that $W_{R,k_0}=f\theta^{\odot 2k_0}$ on $V$. Simply connectedness tells that there is a holomorphic $g$ on $V$ such that $g^{2k_0}=f$. Let $\eta_V=g\theta$ which is decomposable and its $2k_0$-th symmetric power is $W_{R,k_0}$. We call such $V$ a root neighbourhood of $W_{R,k_0}$. 

Let $\{V_\alpha\}$ be a maximal collection within those collections of root neighbourhood where any two root neighbourhoods do not intersect. Let
\[
    U'=\bigcup_{\alpha} V_\alpha.   
\]
It is open and we prove it is dense in $U$ then it is dense in $M$. If $U\backslash \overline{U'}$ is a nonempty open set, then we can find another root neighbourhood outside this collection and it does not intersect with any other root neighbourhood in the collection. This contradicts the maximality of the collection. 
\end{remark}

By Engel's Theorem of complex Lie algebra, we have
\begin{corollary}
    The Chern torsion algebra $\mathfrak{A}_{x,\nabla}$ under assumptions of Proposition~\ref{prop:vanish_characters} is nilpotent at each $x\in M$.
\end{corollary}

\subsection*{Proof of Theorem~\ref{mainthm:CKL-Kahler}} 
When dimension $n=1$, every Hermitian metric is K\"{a}hler. When $n=2$ the balancedness gives $\dif \omega = 0$. Assume that $n\ge 3$.

Suppose that $x_0$ is a point where the holomorphic sectional curvature is positive. And then within a neighbourhood $U$ of $x_0$ the holomorphic sectional curvature is positive.

Suppose that $T_\nabla$ is nonzero at some point $x\in U$.
Since $\mathfrak{A}:=\mathfrak{A}_{x,\nabla}$ is a nonabelian nilpotent Lie algebra, then the lower central series
\[
    \mathfrak{A}=\mathfrak{A}_1 \supset \mathfrak{A}_2 \supset \cdots \supset \mathfrak{A}_s \supset \mathfrak{A}_{s+1} =0 , \quad \mathfrak{A}_{j+1} = [ \mathfrak{A},\mathfrak{A}_j]_\nabla,\quad  s\geq 2,
\]
is finite. Choose $u \in \mathfrak{A}, v\in \mathfrak{A}_{s-1}$ with $w=[u,v]_\nabla\ne 0$. Then $w$ is in the center of $\mathfrak{A}$ and
\[
    \mathfrak{B}_\nabla(u\wedge v\wedge w) = [u,v]_\nabla\odot w + [v,w]_\nabla\odot u + [w,u]_\nabla\odot v = w \odot w.
\]
Since $\hat{R}_\nabla\circ\mathfrak{B}_\nabla=0$ \eqref{eq:bianchi_map}, then
\[
   \HSC(w/|w|)=\frac{1}{|w|^4} R_\nabla(w,\bar{w},w,\bar{w})=\frac{1}{|w|^4}h(\hat{R}_\nabla(w\odot w),(w\odot w)) = 0,
\]
which is a contradiction. So $T_\nabla\equiv 0$ on $U$ and it is holomorphic by Proposition~\ref{prop:basic_identities}. Thus, it vanishes identically on $M$ which is equivalent to $h$ being K\"{a}hler.

\subsection*{Proof of Corollary~\ref{cor:vanish_pq}}

By \cite{ZhangZhang2023}, $M$ is rationally connected and projective. Then using Theorem 1.7 and Remark 1.6 in \cite{LiZhangZhang2024} we know that $(T^{1,0}M, \bar\partial)$ is HN-positive in their notation, and their Corollary 4.4 directly states that the HN-positiveness of $(T^{1,0}M, \bar\partial)$ implies that 
\[
    H^0\!\left(M, (T^{1,0}M)^{\otimes q}\otimes(T^{*1,0}M)^{\otimes p}\right)=0,
\] 
when $p\ge 1, q\ge 0$ and $p\ge Cq$ for some constant $C$ which depends only on $M$, its tangent and cotangent bundle.

\section{A counterexample of Yang-Zheng's Conjecture}
In this section we prove Theorem~\ref{mainthm:counterexample_YZ}.

Consider a compact annulus 
\[
    K=\{z\in\mathbb{C}^2:q\le |z|\le 1\},
\]
and each orbit intersects with $K$ which means the quotient map $\pi$ is surjective on $K$. Thus, $H_q$ is compact.

Define
\[
    A_\epsilon(z) = I + \frac{\epsilon}{|z|^2}\begin{pmatrix} -z_1z_2 & z_1^2 \\ -z_2^2 & z_1z_2 \end{pmatrix}.
\]
And we define a Hermitian metric on $\mathbb{C}^2\backslash\{0\}$,
\[
    h_{\epsilon,z}(\xi,\xi) = \frac{1}{|z|^2}|A_\epsilon(z)\xi|^2_{\text{Euclidean}},\quad \xi\in T^{1,0}_z(\mathbb{C}^2\backslash\{0\})\cong \mathbb{C}^2.
\]
Since
\[
    h_{\epsilon,qz}(q\xi,q\xi) = \frac{1}{|qz|^2}|A_\epsilon(qz)(q\xi)|^2_{\text{Euclidean}} =h_{\epsilon,z}(\xi,\xi),
\]
then this Hermitian metric is well-defined on $H_q$ and we still denote the metric on $H_q$ by $h_\epsilon$. 

For a fixed unitary frame $e$ at one point $p$, RBC-positivity is equivalent to
\[
    Q_p(B):=\sum_{i,j,k,\ell=1}^2 R_{i\bar{j}k\bar{\ell}} B_{ij}B_{k\ell}>0
\]
for every positive semi-definite Hermitian matrices $B$, see \cite[Formula (1)]{YangZheng2019}. Moreover, both $h_{\epsilon,z}$ and $Q_p(B)$ are invariant under $\SU(2)\times \mathbb{R}_{>0}$ action which is transitive, that is for $g\in\SU(2)\times \mathbb{R}_{>0}$ the holomorphic isometry $\Phi_g$ sends $p$ to $p'$ and $U=\dif\Phi_{g,p}$, then
\[
    Q_p(B) = Q_{p'}(UBU^*).
\]
Thus we only need to prove the RBC-positivity at point $p=(1,0)$.

We take a unitary frame at $p=(1,0)$,
\[
    e_1 = \frac{\partial}{\partial z_1},\quad e_2 = -\epsilon\frac{\partial}{\partial z_1}+\frac{\partial}{\partial z_2}.
\]
A direct calculation gives the following curvature components:
\[
    \begin{aligned}
        R_{1\bar{1}1\bar{1}}=-R_{1\bar{1}2\bar{2}}=-R_{1\bar{2}1\bar{2}}&=-R_{1\bar{2}2\bar{1}}=-R_{2\bar{1}1\bar{2}}=-R_{2\bar{1}2\bar{1}}=\epsilon^2, \\
        R_{1\bar{1}1\bar{2}}&=R_{1\bar{1}2\bar{1}}=\epsilon, \\
        R_{1\bar{2}1\bar{1}}=-R_{1\bar{2}2\bar{2}}=R_{2\bar{1}1\bar{1}}&=-R_{2\bar{1}2\bar{2}}=-R_{2\bar{2}1\bar{2}}=-R_{2\bar{2}2\bar{1}}=-\epsilon(1+\epsilon^2), \\
        R_{2\bar{2}1\bar{1}}&=(1+\epsilon^2)^2,\\
        R_{2\bar{2}2\bar{2}}&=1-2\epsilon^2-\epsilon^4.
    \end{aligned}
\]
For any nonzero positive semi-definite Hermitian matrix
\[
    B=\begin{pmatrix} x & a+\sqrt{-1}b \\ a-\sqrt{-1}b & y \end{pmatrix},\quad x\ge 0,\quad y\ge 0,\quad a^2+b^2\leq xy,
\]
we have
\[
    \begin{aligned}
        Q_p(B) =& \epsilon^2 x^2 + (1+\epsilon^2+\epsilon^4)xy+(1-2\epsilon^2-\epsilon^4)y^2-4\epsilon^2a^2 \\ & + 2\epsilon a(2(1+\epsilon^2)y-\epsilon^2x).
    \end{aligned} 
\]
By Young's inequality and $|a|\le \sqrt{xy}$, the last term 
\[
    \begin{aligned}
        2\epsilon a(2(1+\epsilon^2)y-\epsilon^2x) &\ge  -4|\epsilon|(1+\epsilon^2)x^{1/2}y^{3/2} - 2 |\epsilon|^3x^{3/2}y^{1/2}\\ &\ge -\frac{1}{2}\epsilon^2x^2-2\epsilon^4 xy-\frac{1}{2}xy-8\epsilon^2(1+\epsilon^2)^2y^2.
    \end{aligned}
\]
And again by $|a|\le \sqrt{xy}$, we have
\begin{equation}\label{eq:RBC_lower_bound}
        Q_p(B)\ge \frac{1}{2}\epsilon^2x^2 + \left(\frac{1}{2}-3\epsilon^2-\epsilon^4\right)xy+(1-10\epsilon^2-17\epsilon^4-8\epsilon^6)y^2.
\end{equation}
When $0< |\epsilon| \le 1/4$, we have
\[
    Q_p(B) >0,
\]
since $(x,y)\ne (0,0)$. Thus, $h_\epsilon$ is a Hermitian metric with positive RBC on $H_q$. 

Since $\mathbb{C}^2\backslash\{0\}$ deformation retracts to $S^3$ and the deck transformation group of $H_q$ is isomorphic to $\mathbb{Z}$, then $\pi_1(H_q)\cong \mathbb{Z}$.
And $H_q$ is homeomorphic to $S^3\times S^1$ whose $b_1(H_q)=1$, thus $H_q$ cannot admit any K\"{a}hler metric. Since every holomorphic map from $\mathbb{CP}^1$ to $\mathbb{C}^2\backslash\{0\}$ is constant, then there is no rational curve in $H_q$. So $H_q$ cannot be rationally connected.

\begin{remark}
    Even if we assume that $\RBC_\nabla>c>0$ (see \cite[Formula (1)]{YangZheng2019}), $H_q$ is still a counterexample. Since we can prove that in \eqref{eq:RBC_lower_bound},
    \[
        \sum_{i,j,k,\ell=1}^2 R_{i\bar{j}k\bar{\ell}} B_{ij}B_{k\ell}\ge \frac{\epsilon^2}{2}(x+y)^2\ge \frac{\epsilon^2}{2}\tr(B^2).
    \]
    Then by scaling the metric, we can construct a counterexample for any $c>0$.
\end{remark}

\printbibliography

\end{document}